\documentclass[12pt,twoside]{amsart}
\usepackage{amssymb}
\usepackage{verbatim}
\usepackage{amsmath}
\usepackage{bm}
\usepackage{a4wide}
\usepackage[latin1]{inputenc}
\usepackage[T1]{fontenc}
\usepackage{times}
\usepackage{amssymb,latexsym}
\usepackage{enumerate}
\usepackage[stable]{footmisc}
\usepackage{color}
\usepackage[colorlinks,linkcolor=black,citecolor=black,urlcolor=black]{hyperref}

\makeatletter
\newcommand{\sumprime}{\if@display\sideset{}{'}\sum%
            \else\sum'\fi}
\makeatother

\begin{document}

\numberwithin{equation}{section}

\newtheorem{theorem}{Theorem}[section]
\newtheorem{proposition}[theorem]{Proposition}
\newtheorem{conjecture}[theorem]{Conjecture}
\def\theconjecture{\unskip}
\newtheorem{corollary}[theorem]{Corollary}
\newtheorem{lemma}[theorem]{Lemma}
\newtheorem{observation}[theorem]{Observation}
\newtheorem{definition}{Definition}
\numberwithin{definition}{section} 
\newtheorem{remark}{Remark}
\def\theremark{\unskip}
\newtheorem{kl}{Key Lemma}
\def\thekl{\unskip}
\newtheorem{question}{Question}
\def\thequestion{\unskip}
\newtheorem{example}{Example}
\def\theexample{\unskip}
\newtheorem{problem}{Problem}

\thanks{The first two authors are supported by National Natural Science Foundation of China, No. 12271101;  the third author is supported by National Natural Science Foundation of China, No. 12071310.}

\address [Bo-Yong Chen] {School of Mathematical Sciences, Fudan University, Shanghai, 200433, China}
\email{boychen@fudan.edu.cn}

\address [Yuanpu Xiong] {School of Mathematical Sciences, Fudan University, Shanghai, 200433, China}
\email{ypxiong@fudan.edu.cn}

\address [Liyou Zhang] {School of Mathematical Sciences, Capital Normal University, Beijing, 100048, China}
\email{zhangly@cnu.edu.cn}

\title{Equality between the Bergman metric and Carath\'{e}odory metric}
\author{Bo-Yong Chen, Yuanpu Xiong and Liyou Zhang}

\date{}

\maketitle

\begin{center}
\emph{In memory of Professor Qi-Keng Lu}
\end{center}

\begin{abstract}
We present an equality between the Bergman metric and Carath\'{e}odry metric.
\end{abstract}

\section{Introduction}
Let $\Omega$ be a bounded domain in $\mathbb{C}^n$.  Let $B_\Omega(z;X)$ and $C_\Omega(z;X)$ be the Bergman metric and  Carath\'{e}odry metric on $\Omega$ respectively,  i.e.,   for $z\in \Omega$ and holomorphic tangent vector $X$,
\begin{eqnarray}
B_\Omega(z;X) & = & \frac{1}{\sqrt{K_\Omega(z)}}\max\left\{|Xf(z)|: f\in\mathcal{O}(\Omega),\ f(z)=0,\ \|f\|_{L^2(\Omega)}\leq1\right\},\label{eq:B}\\ 
C_\Omega(z;X) & = & \max\left\{|Xf(z)|: f\in\mathcal{O}(\Omega),\ f(z)=0, |f|<1 \right\},\label{eq:C}
\end{eqnarray}
where $K_\Omega$ stands for the Bergman kernel.  Both $B_\Omega$ and $C_\Omega$ are  invariant under biholomorphic mappings.   

A famous result going back to Q.-K. Lu (=\,K. H.  Look) states that 
\begin{equation}\label{eq:B>C}
B_\Omega(z;X) \ge C_\Omega(z;X)
\end{equation}
(see \cite{Lu57,Lu58}).     Burbea \cite{Burbea,Burbea2} and Hahn \cite{Hahn76,Hahn78} gave a surprisingly simple proof  of \eqref{eq:B>C} as follows.   Set $b_z=K_\Omega(\cdot,z)/\sqrt{K_\Omega(z)}$ and take $c_{z,X}\in \mathcal O(\Omega)$ satisfying $c_{z,X}(z)=0$,  $|c_{z,X}|<1$ and $Xc_{z,X}(z)=C_\Omega(z;X)$.  Use $b_z c_{z,X}$ as a candidate for \eqref{eq:B},  
one obtains 
\begin{equation}\label{eq:Hahn}
B_\Omega(z;X) \ge \frac{1}{ \sqrt{K_\Omega(z)} } \frac{ |b_z(z)| |Xc_{z,X}(z)|}{\|b_z c_{z,X} \|_{L^2(\Omega)}}=\frac{C_\Omega(z;X)}{\|b_z c_{z,X}\|_{L^2(\Omega)}}.
\end{equation}
Since $\|b_z c_{z,X} \|_{L^2(\Omega)}< \|b_z\|_{L^2(\Omega)}=1$, we get \eqref{eq:B>C}. Moreover, strict inequality holds when $X\neq0$.

In this note,  we show that  \eqref{eq:Hahn} can be improved to an equality.

\begin{theorem}\label{th:Bergman_Caratheodory}
Let $A^2(\Omega)=L^2(\Omega)\cap \mathcal O(\Omega)$ be the Bergman space on $\Omega$.  Given $z,X$,  define the following closed linear subspace of $A^2(\Omega):$
\[
S_{z,X}:=\{f\in{A^2(\Omega)}:  f(z)=0,\ Xf(z)=0\}.
\]
Let $P_{z,X}$ be  the orthogonal projection from $A^2(\Omega)$ to $S_{z,X}^\perp$,  the orthogonal complement of $S_{z,X}$ in $A^2(\Omega)$.  Then we have
\begin{equation}\label{eq:Bergman_Caratheodory}
B_\Omega (z;X)=\frac{ C_\Omega(z;X)}{\|P_{z,X}(b_z c_{z,X})\|_{L^2(\Omega)}}.
\end{equation}
\end{theorem}

Note that the maximizer $c_{z,X}$ might not be unique. However, \eqref{eq:Bergman_Caratheodory} indicates that $P_{z,X}(b_zc_{z,X})$ is independent of the choices of $c_{z,X}$. Since $\|P_{z,X}(b_z c_{z,X})\|_{L^2(\Omega)}\le \|b_z c_{z,X}\|_{L^2(\Omega)}$, so \eqref{eq:Bergman_Caratheodory} implies \eqref{eq:Hahn}. We shall present two proofs of \eqref{eq:Bergman_Caratheodory}. The first proof is motivated by the original approach of Lu \cite{Lu57} and the second proof is based on the calculus of variations, which actually gives a stronger equality
\begin{equation}\label{eq:B_C_1}
\frac{C_\Omega(z;X)}{B_\Omega(z;X)}b_{z,X}=P_{z,X}(b_zc_{z,X}),
\end{equation}
where $b_{z,X}$ is the unique maximizer in \eqref{eq:B} with $Xb_{z,X}(z)>0$. Moreover, \eqref{eq:B_C_1} can be used to show that strict inequality in \eqref{eq:Hahn} holds for some $z$ and $X$ on simple domains like bidiscs (see Proposition \ref{prop:bidisc_neq}). On the other hand, equality in \eqref{eq:Hahn} holds on  balls (see Proposition \ref{prop:ball_eq}). It is natural to ask the following
\begin{question}
Is $\Omega$ biholomorphic to the unit ball if equality in (1.4) holds for some choice of $c_{z,X}$?
\end{question}

\section{A reproducing formula for derivatives}

 Recall that the Bergman kernel of $\Omega$ is given by  
\begin{equation}\label{eq:OB}
K_\Omega (\zeta,z)=\sum^\infty_{j=0}h_j(\zeta)\overline{h_j(z)}
\end{equation}
where $\{h_j\}^\infty_{j=0}$ is an/any complete orthonormal basis of $A^2(\Omega)$.  Of fundamental importance in Bergman's theory  is the following reproducing formula: 
\begin{equation}\label{eq:reproducing}
f(z)=\int_{\zeta\in\Omega}{f(\zeta)}\overline{K_\Omega(\zeta,z)},\ \ \ \forall\,f\in{A^2(\Omega)},\ z\in\Omega.
\end{equation}
In particular,  we have
\begin{equation}\label{eq:reproducing1}
K_\Omega(z)=\int_{\zeta\in\Omega}|K_\Omega(\zeta,z)|^2, \ \ \ \forall\,z\in\Omega.
\end{equation}

The following reproducing formula for derivatives was used in some references (see,  e.g., \cite{Lu57}).  Since we are unable to find a proof in literature,  it is reasonable to present a proof here.

\begin{proposition}\label{prop:reproducing_deravatives}
For any holomorphic tangent vector $X=\sum^n_{j=1}X_j\partial/\partial{z_j}$ and $z\in \Omega$,  we have 
\begin{equation}\label{eq:reproducing_derivatives}
Xf(z)=\int_{\zeta\in\Omega}{f(\zeta)}\overline{\overline{X}_{z}K_\Omega(\zeta,z)},\ \ \ \forall\,f\in{A^2(\Omega)},
\end{equation}
where 
\[
\overline{X}_zK_\Omega(\zeta,z):=\sum^n_{j=1}\overline{X}_j\frac{\partial{K_\Omega(\zeta,z)}}{\partial\overline{z}_j}.
\]
\end{proposition}

\begin{proof}
For any $f\in \mathcal O(\Omega)$,  we have 
\[
Xf(z)=\lim_{\mathbb C\ni t\rightarrow0}\frac{f(z+tX)-f(z)}{t}.
\]
Here we identify $X$ with a complex vector $(X_1,\cdots,X_n)$ for convenience. Since $K_\Omega(\zeta,z)$ is anti-holormorphic in $z$,  we get 
\[
\overline{X}_z{K_\Omega}(\zeta,z)=\lim_{\mathbb C\ni t\rightarrow0} \frac{K_\Omega(\zeta,z+tX)-K_\Omega(\zeta,z)}{\overline{t}}.
\]
Without loss of generality,   we assume that $|X|=1$.  Let $\delta_\Omega(z)$ denote the Euclidean distance from $z$ to $\partial\Omega$.   For $r=\delta_\Omega(z)/2$ and $ z'\in B(z,r)$,  Cauchy's estimates yield 
\begin{eqnarray}\label{eq:derivative_1}
\left|\overline{X}_z{K_\Omega}(\zeta,z')\right|^2
&\leq& \frac{C_n}{r^{2n+2}}\int_{w\in{B(z',r/2)}}|K_\Omega (\zeta,w)|^2 \nonumber\\
&\leq& \frac{C_n}{\delta_\Omega(z)^{2n+2}}\int_{w\in{B(z,3\delta_\Omega(z)/4)}}|K_\Omega(\zeta,w)|^2
\end{eqnarray}
for all $\zeta\in\Omega$.  
Here $C_n$ denotes a generic constant  depending on $n$.  It follows from \eqref{eq:derivative_1} that for $t\in\mathbb{C}$ with $0<|t|<r$,  
\begin{eqnarray*}
\left|\frac{K_\Omega (\zeta,z+tX)-K_\Omega(\zeta,z)}{\overline{t}}\right|^2
&\leq& \int^1_0\left|\overline{X}_z{K_\Omega}(\zeta,z+stX)\right|^2ds\\
&\leq& \frac{C_n}{\delta_\Omega(z)^{2n+2}}\int_{w\in{B(z,3\delta_\Omega(z)/4)}}|K_\Omega(\zeta,w)|^2.  
\end{eqnarray*}
We claim that
\[
\zeta\mapsto\int_{w\in{B(z,3\delta_\Omega(z)/4)}}|K_\Omega(\zeta,w)|^2
\]
is an integrable function on $\Omega$.  To see this,  simply note that
\begin{eqnarray*}
&& \int_{\zeta\in \Omega}\int_{w\in B(z,3\delta_\Omega(z)/4)}|K_\Omega(\zeta,w)|^2\\
&=& \int_{w\in B(z,3\delta_\Omega(z)/4)}\int_{\zeta\in \Omega} |K_\Omega(\zeta,w)|^2\ \ \ \ \  \text{(by Fubini's Theorem)}\\
&=& \int_{w\in B(z,3\delta_\Omega(z)/4)}K_\Omega(w) \ \ \ \ \  \text{(by \eqref{eq:reproducing1})}\\
& \le & C_n
\end{eqnarray*}
in view of the well-known Bergman inequality.  Since the following family of functions
\[
\left\{\Phi_t(\zeta):=f(\zeta)\overline{\left[\frac{K_\Omega (\zeta,z+tX)-K_\Omega(\zeta,z)}{\overline{t}}\right]}:\ 0<|t|<r\right\}
\]
is dominated by
\[
\Phi(\zeta):=|f(\zeta)|\left[\int_{w\in B(z,3\delta_\Omega(z)/4)}|K_\Omega(\zeta,w)|^2\right]^{1/2}\in{L^1(\Omega)},
\]
it follows immediately from   the dominated convergence theorem that
\[
Xf(z)=\lim_{\mathbb C\ni t\rightarrow0}\frac{f(z+tX)-f(z)}{t}=\lim_{\mathbb C\ni t\rightarrow0}\int_{\zeta\in\Omega}\Phi_t(\zeta)=\int_{\zeta\in\Omega}{f(\zeta)}\overline{\overline{X}_{z}K_\Omega(\zeta,z)}.
\]
\end{proof}

\section{Proof of Theorem \ref{th:Bergman_Caratheodory}}

Recall that
\begin{equation}\label{eq:XXlogK}
B_\Omega(z;X)^2=X\overline{X}\log{K_\Omega(z)}=\frac{X\overline{X}K_\Omega(z)}{K_\Omega(z)}-\frac{|\overline{X}K_\Omega(z)|^2}{K_\Omega(z)^2}.
\end{equation}
By \eqref{eq:OB},  we have
\[
\overline{X}K_\Omega(z)=\overline{X}_z{K_\Omega(\zeta,z)}|_{\zeta=z},\ \ \ \text{and}\ \ \ X\overline{X}K_\Omega(z)=X_\zeta\overline{X}_z{K_\Omega(\zeta,z)}|_{\zeta=z},
\]
where $X_\zeta$ and $\overline{X}_z$ denote derivatives with respect to $\zeta$ and $z$, respectively. Apply the reproducing formulas \eqref{eq:reproducing} and \eqref{eq:reproducing_derivatives} to the holomorphic function
\[
h(\zeta):=\overline{X}_z{K_\Omega(\zeta,z)},
\]
we obtain
\begin{equation}\label{eq:first_order_reproducing}
\overline{X}K_\Omega(z)=h(z)=\int_\Omega{h(\cdot)}\overline{K_\Omega(\cdot,z)}=\int_\Omega{\overline{X}_z{K_\Omega(\cdot,z)}}\overline{K_\Omega(\cdot,z)}
\end{equation}
and
\begin{equation}\label{eq:second_order_reproducing}
X\overline{X}K_\Omega(z)=Xh(z)=\int_\Omega{h(\cdot)}\overline{\overline{X}_z{K_\Omega(\cdot,z)}}=\int_\Omega|h|^2=\int_\Omega\left|\overline{X}_z{K_\Omega(\cdot,z)}\right|^2.
\end{equation}

Set
\[
h_0(\zeta):=b_z(\zeta)=\frac{K_\Omega(\zeta,z)}{\sqrt{K_\Omega(z)}}\ \ \ \text{and}\ \ \ h_1(\zeta):=\frac{P_{z,X}(b_z c_{z,X})(\zeta)}{\|P_{z,X}(b_z c_{z,X} )\|_{L^2(\Omega)}}.
\]
By  \eqref{eq:reproducing},  we conclude that $h_0\perp {S_{z,X}}$.  Moreover,  we have $h_0\perp h_1$,  for  
\begin{eqnarray*}
\int_\Omega{h_1}\overline{h_0} 
&=& \frac{1}{\|P_{z,X}(b_z c_{z,X})\|_{L^2(\Omega)}}\int_\Omega{P_{z,X}(b_z c_{z,X})}\overline{h_0}\\
&=& \frac{1}{\sqrt{K_\Omega(z)}\, \|P_{z,X}(b_z c_{z,X})\|_{L^2(\Omega)} }\int_\Omega b_z(\cdot) c_{z,X}(\cdot) \overline{K_\Omega(\cdot,z)}\\
&=& \frac{b_z(z)c_{z,X}(z)}{\sqrt{K_\Omega(z)}\, \|P_{z,X}(b_z c_{z,X})\|_{L^2(\Omega)} }\\
&=& 0.
\end{eqnarray*}
Since $\dim_{\mathbb{C}}{S_{z,X}^\perp}=2$, it follows that $\{h_0,h_1\}$ is a complete orthonormal basis of $S_{z,X}^\perp$.  For any $f\in{S_{z,X}}$,  we infer from \eqref{eq:reproducing_derivatives} that
\[
\int_\Omega{f\overline{h}}=\int_\Omega{f}\overline{\overline{X}_zK_\Omega(\cdot,z)}=Xf(z)=0,
\]
i.e.,  $h\in S_{z,X}^\perp$.  Thus we may write
\begin{equation}\label{eq:h_decomposition}
h=c_0h_0+c_1h_1,
\end{equation}
where
\[
c_0=\int_\Omega{h}\overline{h_0} = \frac{1}{\sqrt{K_\Omega(z)}}\int_\Omega{h(\cdot)}\overline{K_\Omega(\cdot,z)} = \frac{\overline{X}K_\Omega(z)}{\sqrt{K_\Omega(z)}}
\]
in view of \eqref{eq:first_order_reproducing},  and
\begin{eqnarray*}
\overline{c}_1
= \int_\Omega{h_1}\overline{h} 
&=& \frac{1}{\|P_{z,X}(b_z c_{z,X})\|_{L^2(\Omega)}}\int_\Omega { P_{z,X}(b_z c_{z,X})}\overline{h}\\
&=& \frac{1}{\|P_{z,X}(b_z c_{z,X})\|_{L^2(\Omega)}}\int_\Omega{b_z c_{z,X}}\overline{h}\ \ \ \ \  \text{(since }h\in{S_{z,X}^\perp}\text{)}\\
&=& \frac{1}{\|P_{z,X}(b_z c_{z,X})\|_{L^2(\Omega)}}\int_\Omega b_z c_{z,X}\overline{\overline{X}_zK_\Omega(\cdot,z)}\\
&=& \frac{X(b_z c_{z,X})(z)}{\|P_{z,X}(b_z c_{z,X})\|_{L^2(\Omega)}}\ \ \ \ \  \text{(by \eqref{eq:reproducing_derivatives})}\\
&=& \frac{ \sqrt{K_\Omega(z)} Xc_{z,X}(z) }{ \|P_{z,X}(b_z c_{z,X})\|_{L^2(\Omega)} }\\
&=& \frac{\sqrt{K_\Omega(z)} C_\Omega(z;X)}{ \|P_{z,X}(b_z c_{z,X})\|_{L^2(\Omega)} }.  
\end{eqnarray*}
These together with \eqref{eq:second_order_reproducing} and \eqref{eq:h_decomposition} yield
\begin{eqnarray}
X\overline{X}K_\Omega(z)
&=& \int_\Omega|h|^2
=  |c_0|^2+|c_1|^2\nonumber\\
&=& \frac{|\overline{X}K_\Omega(z)|^2}{K_\Omega(z)}+\frac{K_\Omega(z)\, C_\Omega(z;X)^2 }{\|P_{z,X}(b_z c_{z,X})\|_{L^2(\Omega)}^2}\label{eq:XXK}.
\end{eqnarray}
The assertion follows immediately from \eqref{eq:XXlogK} and \eqref{eq:XXK}.
\section{An alternative approach}

Following Bergman \cite{Bergman70},  we consider the minimizing problem
\begin{equation}\label{eq:minimizer_metric}
m_\Omega(z;X):=\min\left\{\|f\|_{L^2(\Omega)}:f(z)=0,Xf(z)=1\right\}.
\end{equation}
By \eqref{eq:B}, we have
\begin{equation}\label{eq:m}
m_\Omega(z;X)=\frac{1}{\sqrt{K_\Omega(z)}B_\Omega(z;X)}.
\end{equation}
Let $m_{z,X}$ be the unique minimizer for \eqref{eq:minimizer_metric}, so that $m_{z,X}(z)=0$, $Xm_{z,X}(z)=1$ and
\[
\|m_{z,X}\|_{L^2(\Omega)}=m_\Omega(z;X).
\]
Set
\begin{equation}\label{eq:b_z_X}
b_{z,X}:=\frac{m_{z,X}}{\|m_{z,X}\|_{L^2(\Omega)}}=\frac{m_{z,X}}{m_\Omega(z;X)}=\sqrt{K_\Omega(z)}B_\Omega(z;X)m_{z,X}.
\end{equation}
Note that $b_{z,X}$ is essentially a maximizer in \eqref{eq:B}.

Next, we shall use the calculus of variations in a similar way as \cite{CZ}. Let $t\in\mathbb{C}$ and $h\in{S_{z,X}}$, i.e., $h\in{A^2(\Omega)}$, $h(z)=0$ and $Xh(z)=0$. Since the function
\[
J(t):=\|m_{z,X}+th\|_{L^2(\Omega)}^2
\]
attains a minimum at $t=0$, it follows that
\[
\frac{\partial{J}}{\partial{t}}(0)=\frac{\partial{J}}{\partial\overline{t}}(0)=0.
\]
Thus
\begin{equation}\label{eq:reproducing_zero}
\int_\Omega{h}\overline{m_{z,X}}=0,\ \ \ \forall\,h\in{S_{z,X}},
\end{equation}
so that $b_{z,X}\perp{S_{z,X}}$ in view of \eqref{eq:b_z_X}. Consider
\[
\phi:=b_zc_{z,X}-\frac{C_\Omega(z;X)}{B_\Omega(z;X)}b_{z,X}\in{A^2(\Omega)}.
\]
Since $b_{z,X}(z)=c_{z,X}(z)=0$, we have $\phi(z)=0$. Moreover,
\begin{eqnarray*}
X\phi(z)
&=& b_z(z)Xc_{z,X}(z)-\frac{C_\Omega(z;X)}{B_\Omega(z;X)}Xb_{z,X}(z)\\
&=& \sqrt{K_\Omega(z)}C_\Omega(z;X)-\frac{C_\Omega(z;X)}{B_\Omega(z;X)}Xb_{z,X}(z)\\
&=& 0,
\end{eqnarray*}
for $Xb_{z,X}(z)=\sqrt{K_\Omega(z)}B_\Omega(z;X)$. Thus $\phi\in{S_{z,X}}$, so that
\[
0=P_{z,X}\phi=P_{z,X}(b_zc_{z,X})-\frac{C_\Omega(z;X)}{B_\Omega(z;X)}b_{z,X}\in{A^2(\Omega)}
\]
for $b_{z,X}\in{S_{z,X}^\perp}$, that is, \eqref{eq:B_C_1} holds. Since $\|b_{z,X}\|_{L^2(\Omega)}=1$, we conclude the proof.

\section{Examples}
It follows from \eqref{eq:Bergman_Caratheodory} and \eqref{eq:B_C_1} that the strict inequality in \eqref{eq:Hahn} holds if and only if
\begin{equation}\label{eq:strict_ineq}
\frac{C_\Omega(z;X)}{B_\Omega(z;X)}b_{z,X}\neq{b_zc_{z,X}}.
\end{equation}
This can be used to show the following

\begin{proposition}\label{prop:bidisc_neq}
Let $\mathbb{D}^2=\{z\in\mathbb{C}^2:|z_1|<1,|z_2|<1\}$ be the unit bidisc. For $X=X_1\partial/\partial{z_1}+X_2\partial/\partial{z_2}\in\mathbb{C}^2$ with $|X_1|>|X_2|>0$, we have
\[
B_{\mathbb{D}^2}(0;X) > \frac{C_{\mathbb{D}^2}(0;X)}{\|b_0 c_{0,X}\|_{L^2(\mathbb{D}^2)}}.
\]
\end{proposition}

\begin{proof}
Take a complete orthonormal basis $\{h_j\}^\infty_{j=0}$ of $A^2(\mathbb{D}^2)$ as follows
\[
h_0\equiv\frac{1}{\pi},\ \ \ h_1(\zeta_1,\zeta_2)=\frac{\sqrt{3/2}}{\pi}\zeta_1,\ \ \ h_2(\zeta_1,\zeta_2)=\frac{\sqrt{3/2}}{\pi}\zeta_2,
\]
\[
h_j\in{A^2(\mathbb{D}^2)},\ h_j(0)=0,\ \partial{h_j}(0)=0,\ \ \ j\geq3.
\]
Since $K_{\mathbb{D}^2}(\zeta,0)\equiv1/\pi^2$ for all $\zeta\in\mathbb{D}^2$, it follows that
\begin{equation}\label{eq:b_0_bidisc}
b_0(\zeta)=\frac{K_{\mathbb{D}^2}(\zeta,0)}{\sqrt{K_{\mathbb{D}^2}(0)}}\equiv\frac{1}{\pi},\ \ \ \forall\,\zeta\in\mathbb{D}^2.
\end{equation}

Without loss of generality, let $X=(X_1,X_2)$ be a \textit{unit} vector in $\mathbb{C}^2$ with $|X_1|>|X_2|>0$. Note that $b_{0,X}$ is the unique maximizer for the extremal problem
\[
\max\left\{|Xf(0)|:f\in\mathcal{O}(\Omega),\ f(z)=0,\ \|f\|_{L^2(\mathbb{D}^2)}\leq1\right\}
\]
with $Xb_{0,X}(0)>0$. Given $f\in\mathcal{O}(\Omega)$, $f(z)=0$ and $\|f\|_{L^2(\mathbb{D}^2)}\leq1$, we have $f=\sum^\infty_{j=0}c_jh_j$, where $c_0=0$ and $\sum^\infty_{j=1}|c_j|^2\leq1$. It follows from the Cauchy-Schwarz inequality that
\begin{eqnarray*}
|Xf(0)|^2
&=& \left|c_1Xh_1(0)+c_2Xh_2(0)\right|^2\\
&\leq& \left(|c_1|^2+|c_2|^2\right)\left(|Xh_1(0)|^2+|Xh_2(0)|^2\right)\\
&\leq& |Xh_1(0)|^2+|Xh_2(0)|^2.
\end{eqnarray*}
Both inequalities become equalities when
\[
c_1=\frac{\overline{Xh_1(0)}}{\sqrt{|Xh_1(0)|^2+|Xh_2(0)|^2}}=\overline{X_1},\ \ \ c_2=\frac{\overline{Xh_2(0)}}{|Xh_1(0)|^2+|Xh_2(0)|^2}=\overline{X_2}
\]
and $c_j=0$ when $j\geq3$. Thus
\begin{equation}\label{eq:b_0_X_bidisc}
b_{0,X}(\zeta)=\overline{X_1}h_1(\zeta)+\overline{X_2}h_2(\zeta)=\frac{\sqrt{3/2}}{\pi}\left(\overline{X_1}\zeta_1+\overline{X_2}\zeta_2\right)\ \ \ \forall\,\zeta\in\mathbb{D}^2.
\end{equation}

We claim that \eqref{eq:strict_ineq} holds for $z=0$ and $X$ as above, otherwise $c_{0,X}$ can be written as
\[
c_{0,X}(\zeta)=\sqrt{\frac{3}{2}}\frac{C_\Omega(0;X)}{B_\Omega(0;X)}\left(\overline{X_1}\zeta_1+\overline{X_2}\zeta_2\right)=:\alpha\zeta_1+\beta\zeta_2,\ \ \ \forall\,\zeta\in\mathbb{D}^2,
\]
where $\alpha,\beta\neq0$, in view of \eqref{eq:b_0_bidisc} and \eqref{eq:b_0_X_bidisc}, so that
\[
1\geq\sup_{\mathbb{D}^2}|c_{0,X}|=|\alpha|+|\beta|.
\]
This together with the condition $|X_1|>|X_2|>0$ yields that
\[
|Xc_{0,X}(0)|=|\alpha{X_1}+\beta{X_2}|\leq|\alpha|\cdot|X_1|+|\beta|\cdot|X_2|<(|\alpha|+|\beta|)|X_1|\leq|X_1|.
\]
Set $f_0(\zeta):=\zeta_1$. Then $|f_0|<1$ and
\[
|Xf_0(0)|=|X_1|>|Xc_{0,X}(0)|.
\]
This contradicts with the fact that $c_{0,X}$ is a maximizer of \eqref{eq:C}. Thus \eqref{eq:strict_ineq} holds, so that
\[
B_{\mathbb{D}^2}(0;X)= \frac{C_{\mathbb{D}^2}(0;X)}{\|P_{0,X}(b_0 c_{0,X})\|_{L^2(\mathbb{D}^2)}} > \frac{C_{\mathbb{D}^2}(0;X)}{\|b_0 c_{0,X}\|_{L^2(\mathbb{D}^2)}}.
\]
\end{proof}

In contrast to Proposition \ref{prop:bidisc_neq}, we have

\begin{proposition}\label{prop:ball_eq}
Let $\mathbb{B}^2=\{z\in\mathbb{C}^2:|z_1|^2+|z_2|^2<1\}$ be the unit ball in $\mathbb{C}^2$. For any $z\in\mathbb{B}^2$ and $X=X_1\partial/\partial{z_1}+X_2\partial/\partial{z_2}\in\mathbb{C}^2$, there exists a maximizer $c_{z,X}$ of \eqref{eq:C} such that
\[
B_{\mathbb{B}^2}(z;X) = \frac{C_{\mathbb{B}^2}(z;X)}{\|b_z c_{z,X}\|_{L^2(\mathbb{B}^2)}}.
\]
\end{proposition}

\begin{proof}
Let $F:\mathbb{B}^2\rightarrow\mathbb{B}^2$ be a biholomorphic map. Clearly,
\begin{equation}\label{eq:invariance_B_C}
B_{\mathbb{B}^2}(z;X)=B_{\mathbb{B}^2}(F(z);DF(z)X),\ \ \ C_{\mathbb{B}^2}(z;X)=C_{\mathbb{B}^2}(F(z);DF(z)X),
\end{equation}
where $DF(z)$ is the complex Jacobian matrix of $F$ at $z$.  It is easy to verify that
\[
b_z(\zeta)=b_{F(z)}(F(\zeta))\det(DF(\zeta)).
\]
Moreover, a maximizer of $c_{F(z),DF(z)X}$ for $C_{\mathbb{B}^2}(F(z),DF(z)X)$ induces a maximizer $c_{z,X}$ for $C_{\mathbb{B}^2}(z;X)$ by
\[
c_{z,X}(\zeta):=c_{F(z),DF(z)X}(F(\zeta)),
\]
so that $\|b_zc_{z,X}\|_{L^2(\Omega)}=\|b_{F(z)}c_{F(z),DF(z)X}\|_{L^2(\Omega)}$. This together with \eqref{eq:invariance_B_C} imply that we only need to prove the proposition at $z=0$,   for $\mathbb{B}^2$ is a homogeneous domain.

Recall that $S_{0,X}$ is the subspace
\[
\{f\in{A^2(\mathbb{B}^2)}:f(0)=0,\ Xf(0)=0\}.
\]
It suffices to verify
\begin{equation}\label{eq:ball_eq_condition}
b_0c_{0,X}\perp{S_{0,X}}.
\end{equation}
Similarly as above, we take a complete orthonormal basis of $A^2(\mathbb{B}^2)$ as follows
\[
h_0(\zeta)\equiv\frac{\sqrt{2}}{\pi},\ \ \ h_1(\zeta_1,\zeta_2)=\frac{\sqrt{6}}{\pi}\zeta_1,\ \ \ h_2(\zeta_1,\zeta_2)=\frac{\sqrt{6}}{\pi}\zeta_2,
\]
\[
h_j\in{A^2(\mathbb{B}^2)},\ h_j(0)=0,\ \partial{h_j}(0)=0,\ \ \ j\geq3.
\]
The same argument yields
\begin{equation}\label{eq:b_0_ball}
b_0(\zeta)=\frac{K_{\mathbb{B}^2}(\zeta,0)}{\sqrt{K_{\mathbb{B}^2}(0)}}\equiv\frac{\sqrt{2}}{\pi}
\end{equation}
and
\begin{equation}\label{eq:b_0_X_ball}
b_{0,X}(\zeta)=\overline{X_1}h_1(\zeta)+\overline{X_2}h_2(\zeta)=\frac{\sqrt{6}}{\pi}\left(\overline{X_1}\zeta_1+\overline{X_2}\zeta_2\right)
\end{equation}
for all $\zeta\in\mathbb{B}^2$.

Let $X$ be a \textit{unit} vector in $\mathbb{C}^2$. For any $f\in\mathcal{O}(\mathbb{B}^2)$ with $f(0)=0$ and $|f|<1$, we consider the holomorphic function
\[
f_X(t):=f(tX),\ \ \ \forall\,t\in\mathbb{D}.
\]
Since $f_X(0)=0$ and $|f_X|<1$, we have
\[
|Xf(0)|=|f_X'(0)|\leq1
\]
in view of the Schwarz lemma. Thus $C_{\mathbb{B}^2}(0;X)\leq1$. On the other hand, if one defines
\[
\psi(\zeta):=\overline{X_1}\zeta_1+\overline{X_2}\zeta_2,\ \ \ \forall\,\zeta\in\mathbb{B}^2,
\]
then $\psi\in\mathcal{O}(\mathbb{B}^2)$, $\psi(0)=0$ and
\[
|\psi(z)|^2\leq\left(|X_1|^2+|X_2|^2\right)\left(|z_1|^2+|z_2|^2\right)<1.
\]
Moroever,
\[
|X\psi(0)|=|X_1|^2+|X_2|^2=1,
\]
so that $C_{\mathbb{B}^2}(0;X)=1$ and we may take $c_{0,X}=\psi$. This together with \eqref{eq:b_0_ball} and \eqref{eq:b_0_X_ball} imply
\[
b_0c_{0,X}=\frac{\sqrt{3}}{3}b_{0,X}.
\]
\eqref{eq:ball_eq_condition} follows immediately since $b_{0,X}\perp{S_{0,X}}$ by \eqref{eq:reproducing_zero}.
\end{proof}

It is not difficult to see that analogous results hold for polydiscs and balls in $\mathbb{C}^n$ for all $n\geq2$.

\end{document}